\documentclass[12pt]{amsart} 
\usepackage{amssymb} 
\usepackage{amsmath, amscd} 
\usepackage{amsthm}
\usepackage{amsmath} 
\usepackage{comment} 
\usepackage[table,xcdraw]{xcolor} 
\usepackage{ulem} 
\usepackage{tikz}
\usepackage{float} 
\usepackage{graphicx} 
\usepackage{circuitikz} 
\usepackage[colorlinks=true, linkcolor=blue,urlcolor=blue]{hyperref}

\newtheorem{theorem}{Theorem}[section] 
 
\newtheorem{proposition}[theorem]{Proposition} \newtheorem{lemma}[theorem]{Lemma}
 
 \newtheorem{corollary}[theorem]{Corollary} \theoremstyle{definition}
 
\newtheorem{example}[theorem]{Example} \newtheorem{definition}[theorem]{Definition}

\date\today

\begin{document}

\author[Z. Vesali Mahmood]{Zari Vesali Mahmood} 
\address{Department of Mathematics, Tarbiat
Modares University, 14115-111 Tehran Jalal AleAhmad Nasr, Iran} 
\email{v.zarivesali@modares.ac.ir}

\author[A. Moussavi]{Ahmad Moussavi} 
\address{Department of Mathematics, Tarbiat Modares University, 14115-111
Tehran Jalal AleAhmad Nasr, Iran} 
\email{moussavi.a@modares.ac.ir; moussavi.a@gmail.com}	

\author[P. Danchev]{Peter Danchev}
\address{Institute of Mathematics and Informatics, Bulgarian Academy of Sciences, 1113 Sofia, Bulgaria}
\email{danchev@math.bas.bg; pvdanchev@yahoo.com}

\title[weakly $\sqrt{J}U$ rings]{Weakly $\sqrt{J}U$ Rings}

\maketitle

\begin{abstract} We introduce and study the so-called {\it weakly $\sqrt{J}U$ rings} (hereafter abbreviated as {\it $W\sqrt{J}U$ rings} for short), in which every unit is of the form $j+1$ or $j-1$ for some $j$ in $\sqrt{J(R)} : = \{x \in R : x^n \in J(R) \text{ for some } n\ge 1\}$. This class of rings non-trivially generalizes the classes of $\sqrt{J}U$, $UU$, $JU$, $WUU$ and $WJU$ rings, respectively. We investigate their basic properties showing that they are Dedekind-finite, that $M_n(R)$ is never $W\sqrt{J}U$ for $n\ge 2$, and that when $\operatorname{char}(R)>0$ it must be equal to $2^\alpha 3^\beta$ for some $\alpha, \beta \in \mathbb{N}  \cup \left\{ 0 \right\}$. Moreover, for group rings $RG$, we prove that if $RG$ is $W\sqrt{J}U$, then $R$ is $W\sqrt{J}U$ and $G$ is a torsion group. In addition, when $R$ has positive characteristic and $G$ is a locally finite $p$-group, we give a complete characterization like this: $RG$ is a $W\sqrt{J}U$ ring if, and only if, either $R$ is a $\sqrt{J}U$ ring and $G$ is a $2$-group, or $R$ is a $W\sqrt{J}U$ ring with $3\in J(R)$ and $G$ is a $3$-group, or $R\cong R_1\times R_2$ with $R_1$ a $\sqrt{J}U$ ring, $R_2$ a $W\sqrt{J}U$ ring and $G$ a trivial group. 	

Our results substantially improve on recent achievements due to Saini and Udar in Czech. Math. J. (2025).
\end{abstract}

\section{Introduction and Major Concepts}

Throughout this paper, \( R \) denotes an associative ring with identity, which is {\it not} necessarily commutative. We employ the standard notations for \( U(R) \) to be the group of units, \(\operatorname{Nil}(R) \) to be the set of nilpotent elements, \( C(R) \) to be the center, \( \operatorname{Id}(R) \) to be the set of idempotents, and \( J(R) \) to be the Jacobson radical. Likewise, for any positive integer \( n \), the ring of \( n \times n \) matrices over \( R \) is denoted as \( {\rm M}_{n}(R) \), and the ring of \( n \times n \) upper triangular matrices as \( {\rm T}_{n}(R) \). Also, a ring is called \textit{abelian} if all its idempotents are central, i.e., \( \operatorname{Id}(R) \subseteq C(R) \).

In regard to \cite{8}, a ring \( R \) is said to satisfy the \textit{exchange property} if, for every \( a \in R \), there exists an idempotent \( e \in aR \) such that \( 1-e \in (1-a)R \). Besides, in \cite{8} is introduced the class of \textit{clean rings} as those rings in which each element can be written as the sum of a unit and an idempotent. It is well known that every clean ring is exchange, but the converse does {\it not} hold in general; however, for abelian rings it is true. The stronger notion of \textit{strongly clean rings} as defined in \cite{Nic2} requires the unit and the idempotent to commute one another. It is principally known that strongly clean rings properly generalize strongly \(\pi\)-regular rings and have been extensively studied over the past two decades by many authors.

Furthermore, Diesl \cite{diesl} introduced {\it (strongly) nil-clean rings} as those rings in which each element is a sum (with commuting terms) of a nilpotent and an idempotent. As a natural extension, Danchev and McGovern \cite{DMc} defined \textit{weakly nil-clean rings} in the commutative case as those rings in which every element is either a sum or a difference of a nilpotent and an idempotent. The authors gave a complete characterization of such rings and their group rings. Further, Breaz et al. \cite{BDZ} extended the definition to arbitrary rings and obtained a decomposition theorem for abelian weakly nil-clean rings. They also proved that, for a division ring \( D \) and \( n \ge 1 \), \( M_{n}(D) \) is weakly nil-clean if, and only if, either \( D \cong \mathbb{Z}_2 \) or \( D \cong \mathbb{Z}_3 \) with \( n=1 \).

As a further common generalization to the abelian case, Ko\c{s}an and Zhou \cite{KZ} defined \textit{strongly weakly nil-clean rings}, where the nilpotent and the idempotent are required to commute, and extended the structure theory to this wider setting. However, the full description of weakly nil-clean rings in all generality was given in \cite{wuu}.

Some classical concepts needed for our next presentation are these: Recall that a ring is \textit{Boolean} if every element is an idempotent. A ring is \textit{regular} (resp., \textit{unit-regular}) in the sense of von Neumann if, for every \( a \in R \), there exists \( x \in R \) (resp., \( x \in U(R) \)) such that \( axa = a \); it is also known to be \textit{strongly regular} if \( a \in a^{2}R \) for all \( a \in R\). A ring is \textit{semi-regular} if \( R/J(R) \) is regular and idempotents lift modulo \( J(R) \). While semi-regular rings are exchange, the converse is {\it not} valid in general (see \cite{8}).

On the other hand, mimicking Ko\c{s}an et al. \cite{14}, a ring is called \textit{J-clean} if every element is a sum of an idempotent and an element of \( J(R) \) or, equivalently, if \( R/J(R) \) is Boolean and idempotents lift modulo \(J(R) \) (such rings are also called \textit{semi-Boolean} in \cite{25}). Generally, a ring is \textit{weakly semi-Boolean} provided that \( R = J(R) \pm \operatorname{Id}(R) \); these rings may also be termed as \textit{weakly J-clean}. Any semi-Boolean ring is weakly semi-Boolean, and every weakly semi-Boolean ring is clean, but neither of these two implications is reversible (see \cite{danchev1}).

In another vein, a ring \( R \) is named a \textit{UU-ring} if \( U(R) = 1 + \operatorname{Nil}(R) \) (compare with \cite{12}). Later on, Danchev \cite{wuu} introduced the larger class of \textit{weakly UU-rings} (or just {\it WUU-rings} for short) in which $U(R)={\rm Nil}(R)\pm 1$ and proved, among other results, that an exchange ring is WUU if, and only if, it is clean WUU if, and only if, it is strongly weakly nil-clean if, and only
if, \( J(R) \) is nil and \( R/J(R) \) is isomorphic to either a Boolean ring \( B \), or to \( \mathbb{Z}_3 \), or to \( B \times \mathbb{Z}_3 \). He also showed that \( {\rm M}_{n}(R) \) is never a WUU ring for any \( n \ge 2 \).

Since \( \pm 1 + J(R) \subseteq U(R) \) is always fulfilled, it is quite natural to consider rings for which equality holds: thereby, a ring is called a \textit{JU-ring} (or \textit{UJ-ring} in \cite{14}) provided that \( U(R) = 1 + J(R) \) (see \cite{D}). Furthermore, Danchev \cite{wuj} logically generalized this notion to
\textit{weakly JU-rings} (or just {\it WJU-rings} for short) and characterized exchange (resp., clean) WJU-rings in terms of lifting idempotents and the structure of \( R/J(R) \).

In other direction, the set \[ \sqrt{J(R)} = \{ x \in R : x^{n} \in J(R) \text{ for some } n \ge 1 \} \] was
introduced in \cite{wang}; it strictly contains \( J(R) \), but need {\it not} be a subring however. Evidently, \( \operatorname{Nil}(R) \subseteq \sqrt{J(R)} \).

Recently, Saini and Udar \cite{SU} defined a ring to be a \textit{\(\sqrt{J}U\) ring} provided that \( U(R) = 1 + \sqrt{J(R)} \). This class properly contains both UU-rings and JU-rings. They proved that this property passes to corner rings and that \( M_{n}(R) \) is \(\sqrt{J}U\) only when \( n=1 \) and \( R \) itself
is \(\sqrt{J}U\); moreover, \(\sqrt{J}U\) rings are Dedekind-finite.

Independently, Danchev et al. \cite{DDE} studied the same class under the name \textit{UJ\(^\#\) rings} (with \( \sqrt{J(R)} = J^{\#}(R) \)) and investigated their various extensions, including even group rings, in detail.

In the present article, we define and explore the class of \textit{weakly \(\sqrt{J}U\) rings} (abbreviated hereafter as \textit{\(W\sqrt{J}U\) rings} for convenience) via the condition \[ U(R) = \pm 1 + \sqrt{J(R)}. \] Thus, every unit can be written either as \( j+1 \) or as \( j-1 \) for some \( j \in \sqrt{J(R)} \). This notion absolutely expands all of the notions of \(\sqrt{J}U\), UU, JU, WUU and WJU rings, respectively.

We also consider skew polynomial extensions. In fact, for a ring endomorphism \( \alpha : R \to R \), let \( R[x;\alpha] \) designate the {\it ring of skew polynomials} with multiplication given by \( xr = \alpha(r)x \). When \( \alpha = \operatorname{id}_R \), we obtain the ordinary {\it polynomial ring} \( R[x] \). Similarly, \( R[[x;\alpha]] \) designates the {\it ring of skew formal power series} with \( R[[x]] = R[[x;\operatorname{id}_R]] \) being the standard {\it formal power series ring}.

To summarize our current work, in section 2, we introduce weakly \(\sqrt{J}U\) rings and establish their key properties. We prove several basic preliminaries concerning the set \(\sqrt{J(R)}\), as well as we investigate the behavior of the \(W\sqrt{J}U\) property under taking homomorphisms, direct products, corner rings, and matrix rings, and characterize \(W\sqrt{J}U\) rings in terms of their Jacobson radical and unit group. We also examine the connections between \(W\sqrt{J}U\) rings and other classical classes of rings such as clean, \(\pi\)-regular and weakly Boolean rings, respectively.

The subsequent section 3 is devoted to the study of group rings. We provide necessary conditions for a group ring \(RG\) to be \(W\sqrt{J}U\) establishing that then \(R\) must be \(W\sqrt{J}U\) and \(G\) must be a torsion group. Under some additional assumptions on the characteristic of \(R\) or on the group \(G\), we obtain sharper results concluding that \(G\) is either a \(2\)-group or a \(3\)-group in certain cases. The main statement of this section gives a total characterization when a group ring \(RG\) with \(R\) of positive characteristic and \(G\) a locally finite \(p\)-group is weakly $\sqrt{J}U$.

\section{Properties of Weakly $\sqrt{J}U$ Rings}

The aim of this section is to investigate general structure properties of weakly $\sqrt{J}U$ rings and apply them furthermore in the stated structural results. We also examine the transversal between weakly $\sqrt{J}U$ rings and other well-known ring classes such as clean, \(\pi\)-regular, weakly Boolean rings, etc.

\medskip

Our first preliminary technicality is the following.

\begin{lemma}\label{1.2}\cite[Lemma 2.1]{DDE}
Let $R$ be any ring. Then, the following assertions hold:

(1) If $a \in \sqrt{J(R)}$ and $b \in R$ with $ab = ba$, then $ab \in \sqrt{J(R)}$. 	

(2) For every $a \in R$ and $n \in \mathbb{N}$, we have $a^n \in \sqrt{J(R)}$ if and only if $a \in \sqrt{J(R)}$.

(3) If $a \in \sqrt{J(R)}$, then $1 - a \in U(R)$. 	

(4) If $a \in \sqrt{J(R)} \cap C(R)$, then $a \in J(R)$.

(5) For every ideal $I \subseteq J(R)$, we have $\sqrt{J(R/I)} = \sqrt{J(R)}/I$.

(6) If $R = \prod R_i$, then $\sqrt{J(R)} = \prod \sqrt{J(R_i)}$.

(7) If $a, b\in R$ and $ab\in \sqrt{J(R)}$, then $ba\in \sqrt{J(R)}$.
	
(8) $Nil(R) + J(R) \subseteq \sqrt{J(R)}$.
\end{lemma}

It is evident that, for a commutative ring \( R \), $\sqrt{J(R)}$ coincides with the Jacobson radical \( J(R)
\).

\begin{lemma}\label{2.1} If \(f: R \rightarrow S\) is a surjective ring homomorphism, then $f(\sqrt{J(R)})
\subseteq \sqrt{J(S)}$.
\end{lemma}

\begin{proof} Let $t\in f(\sqrt{J(R)})$. Then, $t=f(x)$ where $x\in \sqrt{J(R)}$. So, $x^n\in J(R)$ for some $n$. Hence, $$t^n=(f(x))^n=f(x^n)\in f(J(R))\subseteq J(S)$$ referring to \cite[Ex.4.10]{lam}. Therefore, $t\in \sqrt{J(S)}$, as required.
\end{proof}

We now logically arrive at the following fundamental concept.

\begin{definition} A ring \( R \) is referred to as {\it weakly $\sqrt{J}U$}, abbreviated for further use as $W\sqrt{J}U$, if at least one of the equalities \[ U(R) = \pm1 + \sqrt{J(R)}. \] is fulfilled.

This condition is obviously equivalent to requiring that every unit in \( R \) can be expressed as \( j + 1 \) or \( j - 1 \), where \( j \in \sqrt{J(R)} \).	
\end{definition}

A quick observation that the defined class of rings is {\bf new} and {\it non-trivial} is like this.

\begin{example} The ring $\mathbb{Z}_3$ is $W\sqrt{J}U$, but {\it not} $\sqrt{J}U$.
\end{example}

Recall that a {\it rationally closed} subring of a ring $R$ is such a subring $S$ that possesses the property $U(R) \cap S = U(S)$. It is easy to see that every unital subring has this property.

\begin{proposition}\label{sub} Let $S$ be a rationally closed subring of $R$. If $R$ is a $W\sqrt{J}U$ ring,
then $S$ is also a $W\sqrt{J}U$ ring.
\end{proposition}

\begin{proof} Since $S \cap U(R) = U(S)$, we have \(S \cap J(R) \subseteq J(S)\). Also, it can easily be shown that \(S \cap \sqrt{J(R)} \subseteq \sqrt{J(S)}\), which we leave to the interested reader for a direct check, as requested.
\end{proof}

We proceed by proving a next series of technicalities necessary for our further exploitation.

\begin{proposition}\label{1} Let \(I \subseteq J(R)\) be an ideal of a ring \(R\). Then, \(R\) is $W\sqrt{J}U$
if and only if so is the quotient \(R/I\).
\end{proposition}

\begin{proof} Suppose \(R\) is a $W\sqrt{J}U$ ring and choose \(u + I \in U(R/I)\). Thus, \(u \in U(R)\) and hence \(u = \pm1 + j\), where \(j \in \sqrt{J(R)}\). Therefore, one writes that
\[(u + I) = \pm(1+I)+(j+I),\] where \(j + I \in \sqrt{J(R)}/I = \sqrt{J(R/I)}\) in view of Lemma \ref{1.2}(5).

Conversely, suppose \(R/I\) is a $W\sqrt{J}U$ ring and choose \(u \in U(R)\). Thus, \(u + I \in U(R/I)\) whence \((u + I) = \pm(1 + I) + (j + I)\), where \(j + I \in \sqrt{J(R/I)}=\sqrt{J(R)}/I \). This means that \(u + I = \pm(1 + j) + I\). So, \[u \pm(1 + j) \in I \subseteq J(R).\] Consequently, \(u = \pm1 + j^\prime\), where \(j^\prime \in \sqrt{J(R)}\) in accordance with Lemma \ref{1.2}(8). Finally, \(R\) is a $W\sqrt{J}U$ ring, as expected.
\end{proof}

\begin{proposition}\label{product} Let $R$ be a $W\sqrt{J}U$ ring and $S$ a $\sqrt{J}U$ ring. Then, $R\times S$
is a $W\sqrt{J}U$ ring.
\end{proposition}

\begin{proof} Choose \((u, v) \in U(R \times S) = U(R) \times U(S)\). Thus, there exists \( j \in \sqrt{J(R)} \) such that \( u = 1 + j \) or \( u = -1 + j \). We now distinguish the following two cases:

\medskip

\noindent\textbf{Case I.} Suppose \( u = 1 + j \). Then, since \( v \in U(S) \), there is \( j' \in
\sqrt{J(S)} \) such that \( v = 1 + j' \). Hence, \[ (u, v) = (1, 1) + (j, j'), \] where \( (j, j') \in
\sqrt{J(R \times S)} \) owing to Lemma \ref{1.2}(6).

\medskip

\noindent\textbf{Case II.} Suppose \( u = -1 + j \). Then, since \( v \in U(S) \), there is \( j' \in
\sqrt{J(S)} \) such that \( -v = 1 + j' \), i.e., \( v = -1 + (-j') \). Therefore, \[ (u, v) = -(1, 1) + (j,
-j'), \] where \( (j, -j') \in \sqrt{J(R \times S)} \) according to Lemma \ref{1.2}(6).
\end{proof}

We amend the last assertion to the following more general statement.

\begin{proposition}\label{productt} Let $\{R_i\}_{i \in I}$ be an arbitrary family of rings. Then, the direct product $R=\prod_{i \in I} R_i$ of rings $R_i$ is $W\sqrt{J}U$ if and only if each $R_i$ is $W\sqrt{J}U$ and at most one of them is {\it not} $\sqrt{J}U$, that is, all but one are $\sqrt{J}U$ as the excluded member is proper $W\sqrt{J}U$.
\end{proposition}

\begin{proof} ($\Rightarrow$) Manifestly, every $R_i$ is $W\sqrt{J}U$ as being a homomorphic image. Suppose now there are two members $R_{i_1}$ and $R_{i_2}$ with $i_1 \neq i_2$ that are {\it not} $\sqrt{J}U$. Then, there exist some $u_{i_j} \in U (R_{i_j})$ with $j=1,2$ such that both $u_{i_1} \in U (R_{i_1})$ and $-u_{i_2} \in U (R_{i_2})$ are {\it not} $\sqrt{J}U$ decomposable. Choosing $u=(u_i)$, where $u_i=1$ whenever $i \neq i_{j}$ $j=1,2$, we infer that both $u$ and $-u$ are {\it not} the sum of $1$ and a element from $\sqrt{J}$, as desired to get a contradiction. Consequently, almost all but one $R_i$ are $\sqrt{J}U$ rings as the excluded member has to be $W\sqrt{J}U$.\\ ($\Leftarrow$) Assume that $R_{i_0}$ is a $W\sqrt{J}U$ ring and all of the other $R_i$ are $\sqrt{J}U$. So, thanks to \cite[Proposition 2.4]{SU}, we perceive that $\prod_{i \neq i_{0}} R_i$ is $\sqrt{J}U$. Now, Proposition \ref{product} applies to conclude that $R$ is a $W\sqrt{J}U$ ring, as wanted.
\end{proof}

As two immediate consequences, we yield:

\begin{corollary}\label{pro1} Let $L=\prod_{i \in I} R_i$ be the direct product of rings $R_i \cong R$ and $|I|
\geq 2$. Then, $L$ is a $W\sqrt{J}U$ ring if and only if $L$ is a $\sqrt{J}U$ ring if and only if $R$ is a
$\sqrt{J}U$ ring.
\end{corollary}

\begin{corollary}\label{10} For any $n \geq 2$, the ring $R^n$ is $W\sqrt{J}U$ if and only if $R^n$ is
$\sqrt{J}U$ if and only if $R$ is $\sqrt{J}U$.
\end{corollary}

\begin{proposition}\label{matrix} For any non-zero ring \( R \) and any natural number \( n \geq 2 \), the full
matrix ring \( {\rm M}_n(R) \) is not a $W\sqrt{J}U$ ring.
\end{proposition}

\begin{proof} Since \( {\rm M}_2(R) \) is isomorphic to a corner ring of \( {\rm M}_n(R) \), in conjunction with Proposition~\ref{corner} listed and proved below, it suffices to show that \( {\rm M}_2(R) \) is not a $W\sqrt{J}U$ ring.

To that goal, consider the matrix unit $A = \begin{pmatrix} 0 & 1 \\ 1 & 1 \end{pmatrix}$. Now, one observes that
\[ A + I =
\begin{pmatrix} 0 & 1 \\ 1 & 1 \end{pmatrix} + \begin{pmatrix} 1 & 0 \\ 0 & 1 \end{pmatrix} = \begin{pmatrix} 1
& 1 \\ 1 & 2 \end{pmatrix}, \] which is a unit and hence cannot lie in $\sqrt{J(R)}$.

Furthermore, one sees that \[ A - I = \begin{pmatrix} 0 & 1 \\ 1 & 1 \end{pmatrix} - \begin{pmatrix} 1 & 0 \\ 0 & 1
\end{pmatrix} = \begin{pmatrix} -1 & 1 \\ 1 & 0 \end{pmatrix}, \] which is a unit too, and so it is not in
$\sqrt{J(R)}$ either.

Consequently, one concludes that \( {\rm M}_2(R) \) is not a $W\sqrt{J}U$ ring, as pursued.
\end{proof}

A set $\{e_{ij} : 1 \le i, j \le n\}$ of non-zero elements of $R$ is said to be a {\it system of $n^2$ matrix units} if $e_{ij}e_{st} = \delta_{js}e_{it}$, where $\delta_{jj} = 1$ and $\delta_{js} = 0$ for $j \neq s$. In this case, $e := \sum_{i=1}^{n} e_{ii}$ is an idempotent of $R$ and $eRe \cong {\rm M}_n(S)$, where $$S = \{r \in eRe : re_{ij} = e_{ij}r,~~\textrm{for all}~~ i, j = 1, 2, . . . , n\}.$$

\noindent Recall also that a ring $R$ is said to be {\it Dedekind-finite} if $ab=1$ assures $ba=1$ for any $a,b\in R$. In other words, all one-sided inverses in such a ring are necessarily two-sided inverses.

\begin{proposition}\label{dedekind} Every $W\sqrt{J}U$ ring is Dedekind-finite.
\end{proposition}

\begin{proof}
If we assume on the contrary that $R$ is {\it not} a Dedekind-finite ring, then there exist elements $a, b \in
R$ such that $ab = 1$ but $ba \neq 1$. Setting $e_{ij} := a^i(1-ba)b^j$ and $e :=\sum_{i=1}^{n}e_{ii}$, one knows that there is a non-zero ring $S$ such that $eRe \cong M_n(S)$. However, regarding
Proposition~\ref{corner} alluded to below, $eRe$ is a $W\sqrt{J}U$ ring, whence ${\rm M}_n(S)$ must also be a $W\sqrt{J}U$ ring, thus contradicting Proposition~\ref{matrix}, as suspected.
\end{proof}

\begin{lemma}\label{reduced} Let \(R\) be a $W\sqrt{J}U$ ring. If \(J(R) = (0)\) and every non-zero right ideal
of \(R\) contains a non-zero idempotent, then \(R\) is reduced.
\end{lemma}

\begin{proof} Suppose that the contradiction \(R\) is {\it not} reduced holds. Then, there is a non-zero element \(a \in R\) such that \(a^2 =0\). With \cite[Theorem 2.1]{3} at hand, there is an idempotent \(e \in RaR\) such that \(eRe \cong {\rm M}_2(T)\) for some non-trivial ring \(T\). Now, Proposition \ref{corner} tells us that $eRe$ is a $W\sqrt{J}U$ ring, and hence ${\rm M}_2(T)$ is a $W\sqrt{J}U$ ring. This, however, contradicts Proposition~\ref{matrix}, as asked.
\end{proof}

\begin{proposition}\label{memeber} Let \( R \) be a $W\sqrt{J}U$ ring. Then

(1) \( 3 \in U(R) \iff 2 \in J(R) \).

(2) \( 2 \in U(R) \iff 3 \in J(R) \).

In particular, if \(3 \in J(R)\), then \({\rm Id}(R) = \{0, 1\}\).
\end{proposition}

\begin{proof} (1) Since \( 1 + J(R) \subseteq U(R) \) holds always, the implication is immediate.

For the converse, assume \( 3 \in U(R) \). Then, as \( R \) is $W\sqrt{J}U$, we can write either \[ 3 = 1 + j \quad \text{or} \quad 3 = -1 + j \] for some \( j \in \sqrt{J(R)} \). In the first case, \( j = 2 \in \sqrt{J(R)} \) and so $2\in J(R)$ with the aid of Lemma \ref{1.2}(4), and, in the second case, \( j = 4 \in \sqrt{J(R)} \) and hence $4\in J(R)$. It follows now that \( 2 \in J(R) \) in both cases, as pursued.

(2) Suppose \( 3 \in J(R) \). Since \( 1 + J(R) \subseteq U(R) \) is always true, it follows that \( 2 = -1 + 3 \in U(R) \).

Conversely, assume \( 2 \in U(R) \). Then, as \( R \) is $W\sqrt{J}U$, we may write either
\[ 2 = 1 + j \quad \text{or} \quad 2 = -1 + j \] for some \( j \in \sqrt{J(R)} \). The first case ensures \( j=
1 \in \sqrt{J(R)} \), which is a contradiction. The second case insures \( j = 3 \in \sqrt{J(R)} \) and so $3\in J(R)$, as needed.

As for the second part, let \(e\) be an idempotent element of \(R\). Clearly, \(2e - 1\) is a unit, and hence
either \(2e - 1 = j - 1\) or \(2e - 1 = j + 1\). In the first case, \(2e = j\), and thus \(e = 2^{-1}j\in \sqrt{J(R)} \) implying \(e = 0\) in virtue of \cite[Proposition 2.1]{SU}. In the second case, \(2e = j + 2\), so \(e = 2^{-1}j + 1\in 1+ \sqrt{J(R)}\subseteq U(R)
\) guaranteeing that \(e = 1\), as we need.
\end{proof}

\begin{proposition}\label{mino} A ring \(R\) is $\sqrt{J}U$ if and only if \(R\) is $W\sqrt{J}U$ and \(2 \in
J(R)\).
\end{proposition}

\begin{proof} The necessity follows at once from \cite[Proposition 2.5]{SU}.

For the sufficiency, let \(u \in R\), so we must have \(u = j + 1\) or \(u = j - 1\) for some $j\in \sqrt{J(R)}$. Thus, we write \(u = j - 1 = (j - 2) + 1\), where \(j - 2\in \sqrt{J(R)}\), as required.
\end{proof}

The next claim is crucial.

\begin{lemma}\label{char} Let $R$ be a $W\sqrt{J}U$ ring with $\operatorname{char}(R) > 0$. Then,
$\operatorname{char}(R) = 2^{\alpha}3^{\beta}$ for some $\alpha, \beta \in \mathbb{N}  \cup \left\{ 0 \right\}$.
\end{lemma}

\begin{proof} Firstly, if $n:=\operatorname{char}(R)$ is odd, then $n-1$ is even and a unit in $R$. This forces $2 \in U(R)$. Hence, viewing Proposition~\ref{memeber}, we have $3 \in J(R)$ and $R$ is indecomposable. Now, if
$\operatorname{char}(R)=p_1^{\alpha_1}\cdots p_k^{\alpha_k}$ and all $p_i \neq 3$, then
$(\operatorname{char}(R),3)=1$ forcing $1 \in J(R)$, a contradiction. Thus, at least one $p_i$ must be $3$.
Since $R$ is indecomposable, utilizing the famous Chinese Remainder Theorem, we obtain $\operatorname{char}(R)=3^k$ for some $k \in \mathbb{N}$.

If now $n:=\operatorname{char}(R)$ is even, then either $n=3k+1$, $n=3k-1$ or $n=3k$ for some odd integer $k \in \mathbb{N}$. If, foremost, $n=3k+1$ or $n=3k-1$, the fact that $n1_R \in J(R)$ gives $3k \in U(R)$, whence $3 \in U(R)$. So, Proposition~\ref{memeber} gives $2 \in J(R)$. Now, taking into account Proposition~\ref{mino}, $R$ has to be a $\sqrt{J}U$ ring. If, however, we write $n=p_1^{\alpha_1}\cdots p_m^{\alpha_m}$, again the Chinese Remainder Theorem employs to deduce $$R \cong R_1 \times \dots \times R_m,$$ where $p_i^{\alpha_i}1_{R_i} = 0$ for all $1 \le i \le m$. But, since $R$ is a $\sqrt{J}U$ ring, each
$R_i$ is also a $\sqrt{J}U$ ring demonstrating that $2 \in J(R_i)$ for all $1 \le i \le m$; consequently, $p_i = 2$ for every $i$. Hence, in this case, $\operatorname{char}(R)=2^{\alpha}$ for some $\alpha \in \mathbb{N}$.

If $n=3k$ and $n$ is even, we can write $n=2^{\alpha}3^{\beta}p_1^{\alpha_1}\cdots p_m^{\alpha_m}$. Once again the Chinese Remainder Theorem is applicable to get that \[ R \cong R_1 \times \dots \times R_m, \] where $R_1 := R/2^{\alpha}R$, $R_2 := R/3^{\beta}R$, and $R_i := R/p_i^{\alpha_i}R$ for each $3 \le i \le m$. Finally, Proposition~\ref{productt} leads to $p_i = 2$ for every $3 \le i \le m$, as asked for.
\end{proof}

\begin{proposition}\label{corner} For any \( e \in {\rm Id}(R) \), if the ring \( R \) is $W\sqrt{J}U$, then
the corner ring \( eRe \) is also $W\sqrt{J}U$.
\end{proposition}

\begin{proof} Letting \(u \in U(eRe)\), we have \(u + (1-e) \in U(R)\). So, we may write, $u + 1-e= 1 + j$ or $u + 1-e= -1 + j$ for some $j\in \sqrt{J(R)}$. In the first case, $$j=u-e\in \sqrt{J(R)}\cap eRe=\sqrt{J(eRe)}$$ bearing in mind \cite[Proposition 2.11]{SU}, so that $u=e+j\in e+\sqrt{J(eRe)}$.

In the remaining case, we multiply both sides of the equation \(u + 1 = -1 + j\) by \(e\) from both the left and the right sides. This automatically gives that \(eu = -e + ej\) and \(ue = -e + je\). It is now easy to
verify that \(eu = ue = u\) giving \(ej = je\). Moreover, we have \[ je = jee = eje \in e\sqrt{J(R)}e =
\sqrt{J(eRe)}, \] where the last equality follows from \cite[Proposition 2.11]{SU}. Therefore, \[ u = -e + eje
\in -e + \sqrt{J(eRe)}, \] which is exactly what we wanted to show.
\end{proof}

\begin{proposition}\label{3.26} The following two items hold:

(1) A division ring \(R\) is $W\sqrt{J}U$ if and only if either \(R \cong \mathbb{Z}_2\) or \(R \cong \mathbb{Z}_3\).

(2) A ring \( R \) is local and $W\sqrt{J}U$ if and only if either \( R/J(R) \cong \mathbb{Z}_2 \) or
\(R/J(R) \cong  \mathbb{Z}_3 \).
\end{proposition}

\begin{proof} (1) Since \(R\) is a division ring, one has that $J(R)={\rm Nil}(R)=(0)$ and so $\sqrt{J(R)}=(0)$. Consequently, for any non-zero element $a \in R$, where $R \setminus \{0\} = U(R)$, we obtain $a=\pm 1$, and so $a^2 = 1$ leading to $a^3 = a$. Consulting with the paramount Jacobson's theorem allows us to derive that $R$ is necessarily commutative.

Considering now the polynomial \( f(x) = 1 - x^2 \in R[x] \), one detects that \(f(x)\) can have at most two roots in the multiplicative group \(R^{\ast}\), because \(R\) is a field. Let \(A\) denote the set of all roots of \(f\) in \(R^{\ast}\). Since \(R\) is both a field and a $W\sqrt{J}U$ ring, it follows that, for every non-zero element \(a \in R\), the equality \(a^2 = 1\) holds. Hence, \(R^{\ast} = A\), and therefore \(|R^{\ast}| = |A| \le 2\). Consequently, \(R\) is isomorphic to either one of the finite fields \(\mathbb{Z}_2\) or \(\mathbb{Z}_3\). The converse implication is immediate.

(2) This follows directly from (1) and Proposition~\ref{1}.
\end{proof}

In addition to the definitions of regular and unit-regular rings given above, a ring $R$ is known to be {\it $\pi$-regular} if, for each $a\in R$, $a^n\in a^nRa^n$ for some integer $n\ge 1$. Regular rings are always $\pi$-regular. A ring $R$ is called {\it weakly Boolean} if, for any $a \in R$, either $a$ or $-a$ is an idempotent.

We are now ready to attack the following pivotal result.

\begin{theorem}\label{3.13} Let \(R\) be a ring. Then, the following three assertions are equivalent:

(1) \(R\) is a regular $W\sqrt{J}U$ ring.

(2) \(R\) is a \(\pi\)-regular reduced $W\sqrt{J}U$ ring.

(3) \(R\) is a weakly Boolean ring.
\end{theorem}

\begin{proof} \((1) \Rightarrow (2)\). Since \(R\) is regular, it must be that \(J(R) = (0)\), and every non-zero right ideal contains a non-zero idempotent. So, Lemma \ref{reduced} illustrates that \(R\) is reduced. Also, every regular ring is \(\pi\)-regular.

\((2) \Rightarrow (3)\). Notice that reduced rings are abelian, so \cite[Theorem 3]{badawi} enables us that \(R\) is abelian regular, and hence it is strongly regular. Thus, \(R\) is unit-regular. Likewise, we have \({\rm Nil}(R)=J(R)=\sqrt{J(R)}= 0\). Therefore, for any $u\in U(R)$, we obtain $u=\pm 1$. But, since $R$ is unit-regular, for any \(x \in R\) we can write \(x = ue\) for some \(u \in U(R)\) and \(e \in {\rm Id}(R)\). So, $x=\pm e$. Then, $R$ is weakly Boolean, as claimed.

\((3) \Rightarrow (1)\). It is readily verified that \(R\) is regular. Choosing \(u \in U(R)\), we have \(u = \pm 1\), and thus \(R\) is a $W\sqrt{J}U$ ring, as asserted.
\end{proof}

\begin{corollary}\label{3.14} The following four claims are equivalent for a ring \(R\).

(1)	\(R\) is a regular $W\sqrt{J}U$ ring.

(2) \(R\) is a strongly regular $W\sqrt{J}U$ ring.

(3)	\(R\) is a unit-regular $W\sqrt{J}U$ ring.

(4)	\(R\) is a weakly Boolean ring.
\end{corollary}

A ring $R$ is called {\it semi weakly Boolean} if $R/J(R)$ is weakly Boolean and all idempotents of $R$ lift modulo $J(R)$. Analogously, a ring $R$ is called {\it semi regular} if $R/J(R)$ is regular and all idempotents of $R$ lift modulo $J(R)$.

\medskip

We are now paying attention to establish the following main result.

\begin{theorem}\label{3.16} Let \(R\) be a ring. Then, the following three points are equivalent:

(1) \(R\) is a semi regular $W\sqrt{J}U$ ring.

(2) \(R\) is an exchange $W\sqrt{J}U$ ring.

(3) \(R\) is a semi weakly Boolean ring.
\end{theorem}

\begin{proof} \((1) \Rightarrow (2)\). It is pretty simple, because semi regular rings are always exchange.

\((2) \Rightarrow (3)\). Observe that \cite[Corollary 2.4]{8} may be applied to deduce that \(R/J(R)\) is exchange and idempotents lift modulo \(J(R)\). Also, by Proposition~\ref {1}, \(R/J(R)\) is $W\sqrt{J}U$. So, without loss of generality, it can be assumed that \(J(R) = (0)\). Thus, since \(R\) is an exchange ring, every non-zero one sided ideal contains a non-zero idempotent. Exploiting Lemma \ref{reduced}, \(R\) is reduced and so abelian. Therefore, one infer that \[ {\rm Nil}(R) = J(R) = \sqrt{J(R)} = (0) .\] But, as \(R\) is $WUU$ and exchange, \cite[Corollary 2.15]{wuu} informs us that $R$ is a strongly weakly nil-clean ring. Finally, the statement concludes from \cite[Theorem 3.2]{ch}.
	
\((3) \Rightarrow (1)\). Since $R/J(R)$ is weakly Boolean, it is easily inspected that $R/J(R)$ is simultaneously regular and $W\sqrt{J}U$. Hence, $R$ is both semi regular and $W\sqrt{J}U$ with the usage of Proposition~\ref{1}.
\end{proof}

Three more consequences appear.

\begin{corollary}\label{3.17} Let \( R \) be a $W\sqrt{J}U$ ring. Then, the following are equivalent:

(1) \( R \) is a semi regular ring.

(2) \( R \) is an exchange ring.

(3) \( R \) is a clean ring.
\end{corollary}

\begin{proof} The truthfulness of Theorem \ref{3.16} is a guarantor that (1) \(\Leftrightarrow\) (2). 	

\((3) \Rightarrow (2)\). This is straightforward. 	

\((2) \Rightarrow (3)\). If \( R \) is exchange $W\sqrt{J}U$, then \( R \) is reduced by Lemma \ref{reduced} and hence it is abelian. Therefore, \( R \) is abelian exchange, so it is clean.
\end{proof}

\begin{corollary}\label{3.18} Let \( R \) be a ring. Then, the following are equivalent:

(1) \( R \) is an exchange $W\sqrt{J}U$ ring.

(2) \( R \) is a strongly weakly-nil-clean ring.
\end{corollary}

\begin{proof} \((1) \Rightarrow (2)\). The proof is rather similar to the proof of Theorem~\ref{3.16}, so we drop off details. 	

\((2) \Rightarrow (1)\). If \( R \) is strongly weakly-nil-clean, then it is exchange $WUU$ by
\cite[Corollary 2.15]{wuu}. Since ${\rm Nil}(R)\subseteq \sqrt{J(R)}$, we can conclude that $R$ is a $W\sqrt{J}U$ ring, as required.
\end{proof}

\begin{corollary}\label{3.19} Let \( R \) be an exchange ring. The following are equivalent:

(1) \( R \) is a $W\sqrt{J}U$ ring.

(2) \( R \) is a $WUU$ ring.
\end{corollary}

\begin{proof} \((2) \Rightarrow (1)\). It is pretty clear, because ${\rm Nil}(R)\subseteq \sqrt{J(R)}$ always.

\((1) \Rightarrow (2)\). If \( R \) is exchange and $W\sqrt{J}U$, then ${\rm Nil}(R)=J(R)=\sqrt{J(R)}=(0)$ follows from Theorem~\ref{3.16}. Hence, $R$ is $WUU$, as requested.
\end{proof}

We now have all the instruments necessary to prove the following a bit curious affirmation.

\begin{theorem}\label{m} Let $R$ be a ring. Then, $R$ is a $W\sqrt{J}U$ ring if and only if $R/J(R)$ is a
$WUU$ ring.
\end{theorem}

\begin{proof} Suppose that $R$ is $W\sqrt{J}U$ and choose $u+J(R)\in U(R/J(R))$. So, $u\in U(R)$ and hence we write $u=\pm 1+j$, where $j\in \sqrt{J(R)}$. Thus, $$u+J(R)=\pm (1+J(R)) + (j+J(R)),$$ where $j+J(R)\in
{\rm Nil}(R/J(R))$; in fact, since $j\in \sqrt{J(R)}$, we have $j^n\in J(R)$ for some $n$.

Conversely, assume that $R/J(R)$ is $WUU$ and choose $u\in U(R)$. Hence, $u+J(R)\in U(R/J(R))$. Then, $$u+J(R)=\pm(1+J(R))+(q+J(R)),$$ where $q+J(R)\in {\rm Nil}(R/J(R))$ whence $(q+J(R))^n=J(R)$ for some $n$. Therefore, $q^n\in J(R)$, i.e., $q\in \sqrt{J(R)}$. On the other hand, $u\pm(1+q)\in J(R)$ and so $u=\pm1+q+j$, where $j\in J(R)$. But, $q+j\in \sqrt{J(R)}$ adapting Lemma~\ref {1.2}(8). Consequently, $R$ is $W\sqrt{J}U$, as needed.
\end{proof}

Two next consequences sound thus:

\begin{corollary} Let $R$ be a ring with $J(R)=(0)$. Then, $R$ is a $W\sqrt{J}U$ ring if and only if $R$ is a
$WUU$ ring.
\end{corollary}

\begin{corollary} Let $R$ be a ring with $J(R)$ nil. Then, $R$ is a $W\sqrt{J}U$ ring if and only if $R$ is a
$WUU$ ring.
\end{corollary}

The following construction reveals the difficulty in obtaining relationships between the studied classes of rings.

\begin{example} (1) Every $WUJ$ ring is necessarily a $W\sqrt{J}U$ ring. However, the reverse relation
does {\it not} always hold. Indeed, let $R$ be the $\mathbb{F}_2$-algebra generated by $x,y$ satisfying $x^2=0$. Then, $R$ is $W\sqrt{J}U$, but it is {\it not} $WUJ$, because $J(R)=(0)$.

(2) It is not too hard to see that every $WUU$ ring is a $W\sqrt{J}U$ ring. However, the opposite relation does {\it not} necessarily hold. To substantiate this, consider $R := \mathbb{F}_2[[x]]$. So, $R$ is a $W\sqrt{J}U$ ring that is not a $WUU$ ring, because $J(R)$ is not nil (cf. \cite{wuu}).
\end{example}

As usual, let ${\rm Nil}_{*}(R)$ denote the {\it prime radical} (or, in other words, the {\it lower nil-radical}) of a ring $R$, i.e., the intersection of all prime ideals of $R$. We know that ${\rm Nil}_{*}(R)$ is a nil-ideal of $R$. It is long known that a ring $R$ is called {\it $2$-primal} if its lower nil-radical ${\rm Nil}_{*}(R)$ consists precisely of all the nilpotent elements of $R$, that is, ${\rm Nil}_{*}(R)={\rm Nil}(R)$. For instance, it is well known that both reduced rings and commutative rings are  $2$-primal. For an
endomorphism $\alpha$ of a ring $R$, $R$ is called {\it $\alpha$-compatible} if, for any $a,b\in R$,
$ab=0\Longleftrightarrow a\alpha (b)=0$ and, in this case, $\alpha$ is apparently injective.

\medskip

We are now having at disposal all the ingredients necessary to establish the following surprising statement.

\begin{theorem} Let \( R \) be a \(2\)-primal ring, and let \( \alpha \) be an endomorphism of \( R \). If \( R \) is \( \alpha \)-compatible, then the following four issues are equivalent:

(1) $R$ is a $WUU$ ring. 	

(2) $R[x; \alpha]$ is a $W\sqrt{J}U$ ring. 	

(3) $R[x; \alpha]$ is a $WJU$ ring.
	
(4) $R[x; \alpha]$ is a $WUU$ ring.
\end{theorem}

\begin{proof} (2) $\Leftrightarrow$ (3). It follows from \cite[Corollary 3.14]{DDE}. 	

(1) $\Rightarrow$ (2). Let us assume $f=\sum_{i=0}^{n}a_ix^i \in U(R[x; \alpha])$. Since $R$ is a $2$-primal ring, \cite[Corollary 2.14]{Chenpr} works to get $a_0 \in U(R)$ and $a_i \in {\rm Nil}(R)$ for all $1 \le i \le n$. However, as $R$ is a $WUU$ ring, we find that \( a_0 = \pm1 + q \), where \( q \in {\rm Nil}(R) \). Now, we subsequently have $$f = (\pm1 + q) + \sum_{i=1}^{n}a_ix^i \in$$
$$\pm1 + {\rm Nil}(R) + {\rm Nil}(R)[x; \alpha]x=\pm1+{\rm Nil}(R)[x; \alpha]=\pm1+\sqrt{J(R[x; \alpha])},$$
as requested.	

(2) $\Rightarrow$ (1). Let us assume $u \in U(R)$. Then, we can write that $$u \in \pm1 + \sqrt{J(R[x; \alpha])} = \pm1 + {\rm Nil}(R)[x; \alpha].$$ Therefore, $u \pm 1 \in {\rm Nil}(R)$. 	

(1) $\Rightarrow$ (4). Arguing as in (1) $\Rightarrow$ (2), we are done.

(4) $\Rightarrow$ (1). As $R[x; \alpha]/(x)\cong R$ and all units of $R[x; \alpha]/(x)$ lifted to units $R[x; \alpha]$, the implication is sustained.
\end{proof}

What we can extract as a valuable consequence is the following.

\begin{corollary} Let $R$ be a $2$-primal ring. Then, the following four conditions are equivalent:

(1) $R$ is a $WUU$ ring.

(2) $R[x]$ is a $W\sqrt{J}U$ ring.

(3) $R[x]$ is a $WJU$ ring.

(4) $R[x]$ is a $WUU$ ring.
\end{corollary}

We are now concentrating on the following.

\begin{proposition}\label{3.9} A ring $R[[x; \alpha]]$ is $W\sqrt{J}U$ if and only if so is $R$.
\end{proposition}

\begin{proof} Putting $I := R[[x; \alpha]]x$, we then check that $I$ is an ideal of $R[[x; \alpha]]$. Note also
that $J(R[[x; \alpha]])=J(R)+I$, so that $I\subseteq J(R[[x; \alpha]])$. And since $R[[x; \alpha]]/I\cong R$, the result follows with Proposition~\ref{1} in hand.
\end{proof}

We are now concerning with the triangular matrices.

\begin{proposition}\label{uper} Let $R$ be a ring. Then, the following four claims are equivalent:

(1) $R$ is a $\sqrt{J}U$ ring.

(2) ${\rm T}_{n}(R)$ is a $\sqrt{J}U$ ring for all $n \in \mathbb{N}$.

(3) ${\rm T}_n(R)$ is a $\sqrt{J}U$ ring for some $n \in \mathbb{N}$.

(4) ${\rm T}_n(R)$ is a $W\sqrt{J}U$ ring for some $n \geq 2$.
\end{proposition}

\begin{proof} Points (1), (2), and (3) are all equivalent via \cite[Proposition 3.19(3)]{DDE}.

(1) $\Rightarrow$ (4). This is quite trivial, so we omit details.

(4) $\Rightarrow$ (1). Setting $I:=\{ (a_{ij})\in {\rm T}_{n}(R)| a_{ii}=0\}$, we obtain that $I\subseteq J({\rm T}_{n}(R))$ with ${\rm T}_{n}(R)/{I}\cong R^{n}$. Therefore, Proposition~\ref{1} and Corollary~\ref{10} are applicable in combination to get the desired result.
\end{proof}

Let $R$ be a ring and $M$ a bi-module over $R$. The {\it trivial extension} of $R$ and $M$ is stated as \[{\rm T}(R,M) = \{(r, m) : r \in R \text{ and } m \in M\} \] with addition defined component-wise and multiplication
defined by the equality \[ (r, m)(s, n) = (rs, rn + ms). \] One knows that ${\rm T}(R, M)$ is isomorphic to the subring \[ \left\{ \begin{pmatrix} r & m \\ 0 & r \end{pmatrix} : r \in R \text{ and } m \in M \right\} \]
consisting of all the formal $2 \times 2$ matrix ring $\begin{pmatrix} R & M \\ 0 & R \end{pmatrix}$. 

We also notice to be known is that the set of units of ${\rm T}(R, M)$ is precisely \[ U({\rm T}(R, M)) = {\rm T}(U(R), M). \]

We close this section with the following assertion.

\begin{proposition}\label{3.4} The trivial extension ${\rm T}(R, M)$ is a $W\sqrt{J}U$ ring if and only if $R$ is a $W\sqrt{J}U$ ring. 
\end{proposition} 

\begin{proof} Set $A:={\rm T}(R, M)$ and consider the ideal $I:={\rm T}((0), M)$ of $A$. Then, one discovers that $I\subseteq J(A)$ such that $A/I \cong R$. So, our claim follows at once from Proposition \ref{1}, as expected.
\end{proof}

\section{Group Rings}

Let $R$ be a ring and $G$ a group. As usual, the notation $RG$ stands for the group ring as being a module over
$R$ with elements of $G$ as a basis. The homomorphism $\varepsilon \colon RG \to R$, defined by $\sum_{g\in G} r_g g \mapsto \sum_{g\in G} r_g$, is standardly known as the \textit{augmentation homomorphism} of $RG$. Its kernel, $\ker(\varepsilon)$, referred to as the \textit{augmentation ideal} of $RG$, is denoted by $\Delta(RG)$, and equals to $$\Delta(RG)=\left\{ \sum_{g \in G}a_g(1 - g) \colon 1 \neq g \in G, a_g \in G \right\}.$$

Traditionally, a group $G$ is called \textit{locally finite}, provided that any subgroup generated by a finite
subset of $G$ is itself finite. When $p$ is a prime, a $p$-group means that every its element has order equal to a power of $p$. If all non-identity elements of a group have order exactly $p$, the group is said to have
\textit{exponent} $p$. The notation $C_n$ represents the classical cyclic group having only $n$ elements.

\medskip

Before stating and proving the main result of this section, we need to establish four technical machineries.

\begin{lemma}\label{l1} If $RG$ is a $W\sqrt{J}U$ ring, then $R$ is also a $W\sqrt{J}U$ ring. 
\end{lemma}

\begin{proof} Assume $a \in R$. Since $RG$ is a $W\sqrt{J}U$ ring, there is $j \in \sqrt{J(RG)}$ such that
$a = \pm 1 + j$, whence $a = \pm 1 + \varepsilon(j)$. But, Lemma \ref{2.1} reaches us that $\varepsilon(j) \in \sqrt{J(R)}$, as required.
\end{proof}

It is worthy of noticing that $R$ is an epimorphic image of $RG$ as shown above, so that the conclusion could be derived directly too.

\begin{lemma} If $RG$ is a $W\sqrt{J}U$ ring, then $G$ is a torsion group. 
\end{lemma} 

\begin{proof} We prove the claim by contradiction. To that target, assume that there is $g \in G$ which is {\it not} of finite order. Thus, by assumption, we must have either $1-g \in \sqrt{J(RG)}$ or $1+g \in \sqrt{J(RG)}$. So, by Lemma \ref{1.2}(1), we have that either $g(1-g) \in \sqrt{J(RG)}$ or $g(1+g) \in \sqrt{J(RG)}$. Again Lemma \ref{1.2}(3) teaches us that either $1+g-g^2 \in U(RG)$ or $1+g+g^2 \in U(RG)$. 

We only consider the case $1 + g - g^2 \in U(RG)$. Consequently, there are integers $n < m$ and elements $a_i$ with $a_n \neq 0 \neq a_m$ such that \[ (1 + g - g^2)\sum_{i=n}^{m} a_ig^i = 1. \] After expanding the sum, this obviously leads us to a contradiction. 

We process by analogy with $1+g+g^2 \in U(RG)$ to reach a contradiction. Therefore, all elements of $G$ must be of finite order, which means that $G$ is a torsion group, as formulated.
\end{proof}

\begin{lemma}\label{2gr} Suppose that $RG$ is a $W\sqrt{J}U$ ring with $2 \in J(R)$. Then, $G$ is a $2$-group.
\end{lemma} 

\begin{proof} Since $2 \in J(R)$, we obtain $3 \in U(R) \subseteq U(RG)$; therefore, using 
Proposition~\ref{memeber}, $2 \in J(RG)$. Hence, Proposition~\ref{mino} gives that $RG$ is a $\sqrt{J}U$ ring, and so \cite[Theorem 4.4]{uq} shows that $G$ is a $2$-group, as asserted. 
\end{proof}

\begin{lemma}\label{3gr} Suppose that $RG$ is a $W\sqrt{J}U$ ring with $3 \in J(R)$ and that $G$ is a $p$-group. Then $G$ is a $3$-group. 
\end{lemma} 

\begin{proof} Analogically to the previous lemma, it can be shown that $3 \in J(RG)$. 

We now examine the following two basic cases:

\medskip

\noindent{\bf Case 1:} If $p = 2$, we first prove that, for every $g \in G$, $g^2 = 1$. We do that in a way of contradiction. To this purpose, let $n > 1$ be the smallest integer such that $g^{2^n} = 1$. Since $g \in U(RG)$, we have either $1 - g^{2^{n-1}} \in \sqrt{J(RG)}$ or $1 + g^{2^{n-1}} \in \sqrt{J(RG)}$. Because of the relationship $3 \in J(RG)$, it follows that either $1 + 2g^{2^{n-1}} \in \sqrt{J(RG)}$ or $- 2g^{2^{n-1}} \in \sqrt{J(RG)}$. So,
Lemma~\ref{1.2}(1) refers to either \[ 1 + g^{2^{n-1}} = -g^{2^{n-1}} + \sqrt{J(RG)} \in U(RG), \] or \[ 1 - g^{2^{n-1}} = g^{2^{n-1}} + \sqrt{J(RG)} \in U(RG). \] 

\noindent But, on the other hand, we have $(1 + g^{2^{n-1}})(1 - g^{2^{n-1}}) = 0$. This, however, suggests
either $g^{2^{n-1}} = 1$ or $g^{2^{n-1}} = -1$. Since $n$ was chosen minimal with $g^{2^n}=1$, the case
$g^{2^{n-1}} = 1$ cannot occur. Thus, if $g^{2^{n-1}} = -1$, then we detect $$(1 + g^{2^{n-2}})(1 - g^{2^{n-2}}) = 1 - (-1) = 2 \in U(RG),$$ hence $1 + g^{2^{n-2}} \in U(RG)$ and $1 - g^{2^{n-2}} \in U(RG)$, which is the wanted contradiction, because $RG$ is a $W\sqrt{J}U$ ring. Therefore, we must have $g^2 = 1$ for every $g \in G$.

Now, we manage to show that this is also impossible. In fact, if $g^2 = 1$, then $e = \frac12 (1 + g)$ is an idempotent in $RG$. However, Proposition~\ref{memeber} proposes that $RG$ has no non-trivial idempotents, so either $e = 0$ or $e = 1$, which gives either $g = 1$ or $g = -1$. But, since $g \in G$ was arbitrary, we infer $G \cong C_2$, and consequently $RG \cong R \times R$. This, however, contradicts Corollary~\ref{pro1}, because $R$ is not a $\sqrt{J}U$ ring, as $2 \not\in J(R)$, thus substantiating our inspection.

\medskip

\noindent{\bf Case 2:} If $p \neq 2$, we show that $p = 3$. To that end, let $g^{p^k} = 1$. Since $p^k$ is odd, we can write $p^k = 2n + 1$ for some $n \in \mathbb{N}$. Hence, either $1 + g^{2n} \in \sqrt{J(RG)}$ or $1 - g^{2n} \in \sqrt{J(RG)}$.

If, firstly, $1 + g^{2n} \in \sqrt{J(RG)}$, then in conjunction with Lemma~\ref{1.2}(1) multiplying by $g$ gives $1 + g \in \sqrt{J(RG)}$. Moreover, as $3 \in J(RG)$, we deduce $1 - 3 + g \in \sqrt{J(RG)}$ yielding $1 - g \in -1 + \sqrt{J(RG)} \subseteq U(RG)$. Furthermore, since $g^{p^k} = 1$, we extract $(1 - g)(1 + g + \dots + g^{\,p^k-1}) = 0$; but, because $1 - g \in U(RG)$, we get $f := 1 + g + \dots + g^{\,p^k-1} = 0$, so $0 = \varepsilon(f) = p^k$ implying $p \in J(R)$. But, $3 \in J(R)$ amounts to $p = 3$.

If now $1 - g^{2n} \in \sqrt{J(RG)}$, then in virtue of Lemma~\ref{1.2}(1) multiplying by $g$ gives $1 - g \in \sqrt{J(RG)} \cap R\langle g \rangle \subseteq \sqrt{J(R\langle g \rangle)}$. Since $1 - g \in C(R\langle g \rangle) \cap \sqrt{J(R\langle g \rangle)}$, Lemma~\ref{1.2}(4) insures $1 - g \in J(R\langle g \rangle)$, whence $\Delta(R\langle g \rangle) \subseteq J(R\langle g \rangle)$. Consequently, \cite[Proposition 15(i)]{con} tells us that $\langle g \rangle$ is a $q$-group with $q \in J(R)$; but, because, $3 \in J(R)$, we finally obtain $q = 3$, as desired.
\end{proof}

We are now able to prove our final main result.

\begin{theorem} Let $R$ be a ring of positive characteristic, and let $G$ be a locally finite $p$-group. Then,
$RG$ is a $W\sqrt{J}U$ ring if and only if one of the following holds: 

\begin{enumerate} 
\item $R$ is a $\sqrt{J}U$ ring and $G$ is a $2$-group. 
\item $R$ is a $W\sqrt{J}U$ ring with $3 \in J(R)$, and $G$ is a $3$-group. 
\item $R \cong R_1 \times R_2$, where $R_1$ is a $\sqrt{J}U$ ring, $R_2$ is a $W\sqrt{J}U$ ring, and $G$ is the trivial group.
\end{enumerate} 
\end{theorem} 

\begin{proof} Assume foremost that $RG$ is a $W\sqrt{J}U$ ring. Employing Lemma~\ref{char}, we deduce $\operatorname{Char}(R)=2^{\alpha}3^{\beta}$, where $\alpha$ and $\beta$ are positive integers.

If $\alpha = 0$, then $3 \in J(R)$. Hence, combining Lemmas~\ref{l1} and \ref{3gr}, $R$ is a $W\sqrt{J}U$ ring and $G$ is a $3$-group.

If $\beta = 0$, then $2 \in J(R)$. Hence, combining Lemmas~\ref{l1} and \ref{2gr}, $R$ is a $W\sqrt{J}U$ ring and $G$ is a $2$-group. Moreover, Proposition~\ref{mino} guarantees that $R$ is, actually, a $\sqrt{J}U$ ring.

If, however, both $\alpha, \beta \neq 0$, then by the Chinese Remainder Theorem we obtain $R \cong R_1 \times R_2$, where $2^{\alpha}=0$ in $R_1$ and $3^{\beta}=0$ in $R_2$. Therefore, a combination of Corollary~\ref{pro1} and Proposition~\ref{mino}, $R_1$ is a $\sqrt{J}U$ ring and $R_2$ is a $W\sqrt{J}U$ ring. 

On the other hand, since $RG$ is a $W\sqrt{J}U$ ring, the standard isomorphism $(R_1 \times R_2)G \cong R_1G \times R_2G$ shows that the direct product $R_1G \times R_2G$ is also a $W\sqrt{J}U$ ring. Hence, Corollary~\ref{pro1} ensures that $R_1G$ and $R_2G$ are $W\sqrt{J}U$ rings. Now, Lemmas~\ref{3gr} and \ref{2gr} assure that $G$ to be the trivial group, as pursued.

Conversely, assume that (1) holds. Since $R$ is a $W\sqrt{J}U$ ring, Proposition~\ref{mino} gives $2 \in J(R)$.
Also, because $G$ is a locally finite $2$-group, we have $\Delta(RG) \subseteq J(RG)$ in view of \cite[Lemma 2]{zhou}. But, as $R/\Delta(RG) \cong R$, Proposition~\ref{1} terminates that $RG$ must be a $W\sqrt{J}U$ ring.

The proof of point (2) is similar to that of point (1), so it is removed.

The final point (3) is quite obvious, so without treating. This concludes our arguments after all.
\end{proof}

\medskip
\medskip

\section*{Data availability}
No data was used for the research described in the article.

\medskip
\medskip

\section*{Declarations}
The authors declare no any conflict of interests while writing and preparing this manuscript.

\end{document}